\documentclass[11pt,a4paper]{amsart}
\usepackage{amssymb,amsfonts}
\usepackage[T1]{fontenc}
\usepackage{textcomp}
\usepackage{lmodern}
\usepackage{microtype}
\usepackage[pdftex]{graphicx}
\usepackage[margin=3cm]{geometry}
\newcommand{\C}{\mathbb{C}}
\newcommand{\N}{\mathbb{N}}

\newcommand{\R}{\mathbb{R}}

\newcommand{\Z}{\mathbb{Z}}

\newcommand{\cfct}{\mathbf{c}}
\newcommand{\dfct}{\mathbf{d}}
\newcommand{\romannum}{\renewcommand{\labelenumi}{\textnormal{(\roman{enumi})}}}

\newcommand{\bra}{\left<\alpha\right>}
\newcommand{\brb}{\left<\beta\right>}

\usepackage{hyperref}
\usepackage{tabularx}
\newtheorem{theorem}{Theorem}[section]
\newtheorem{proposition}[theorem]{Proposition}

\newtheorem{lemma}[theorem]{Lemma}

\theoremstyle{definition}

\newtheorem{remark}[theorem]{Remark}       
\newtheorem{definition}[theorem]{Definition}  

\title[Non-Integrable Multipliers for the Jacobi Transform]%
{On a Class of Non-Integrable Multipliers for the Jacobi Transform}
\author{Troels Roussau Johansen}
\date{\today}
\address{Mathematisches Seminar\\
   Christian-Albrechts Universit\"at zu Kiel\\
   Ludewig-Meyn-Strasse 4, D-24098 Kiel\\
   Germany}
\email{johansen@math.uni-kiel.de}

\begin{document}

\begin{abstract}
We show that a bounded function $m$ on $\R$ not necessarily integrable at infinity may still yield $L^p$-bounded convolution operators for the Jacobi transform if the nontangential boundary values of $\omega\cdot m$ along the edges of a certain strip in $\C$ yield Euclidean Fourier multipliers, where $\omega$ is a function of the form $\omega(\lambda)=(\lambda^2+4\rho^2)^{\alpha+1/4}$. This partially generalizes similar results by Giulini, Mauceri, and Meda (on rank one symmetric spaces) and Astengo (on Damek--Ricci spaces).
\end{abstract}

\maketitle

\section{Introduction and statement of result}
The study of translation invariant operators has played a decisive role in the development of Euclidean harmonic analysis, as evidenced, for example, by the landmark paper \cite{Hormander-translation} by H\"ormander. A close connection between said translation invariant operators, the Fourier transform, and distributions was uncovered, as such operators turned out to be Fourier multiplier operators, or, what amounts to the same thing, convolution operators with suitable kernels. It didn't take long for the experts to seek new venues for their inquiries. One of the first was the important paper \cite{ClercStein}, where exciting non-Euclidean phenomena were uncovered, in the setting of noncommutative harmonic analysis on a noncompact symmetric space.

Let us specialize to the rank one situation for the moment and define $\Omega_p=\{\lambda\in\C\,:\,\vert\mathrm{Im} \lambda\vert < \vert 2/p-1\vert\rho\}$, where $\rho$ is a certain constant associated with the symmetric space $G/K$ (half the sum of positive roots, see Section \ref{sec.Jacobi}).  Clerc and Stein observed that an $L^p$-multiplier for the so-called spherical transform of a Riemannian symmetric space of the noncompact type always has a holomorphic extension to the strip $\Omega_p$. Several multiplier results followed the publication of \cite{ClercStein} and while we cannot adequately recount the complete literature, let us at least mention \cite{Stanton-Tomas} (the rank one case) and \cite{Anker-Annals} for the general rank case (where the strip $\Omega_p$ is replaced by a tube domain $T_p$ over a certain cone in the dual of the Lie algebra of the Iwasawa-group $A$ in $G$).  More recent advances include \cite{Ionescu}, as well as \cite{BloomXu}.  The latter establishes the results from \cite{Anker-Annals} in the context of {C}h\'ebli-{T}rim\`eche hypergroups (which subsumes the spherical analysis on a rank one symmetric space, and more generally the Jacobi analysis we are dealing with).

It is well-known that an $L^p$-multiplier $m$ for the spherical transform on $G/K$ is determined by its boundary value on the edge on $\Omega_p$, and Anker showed that if this boundary value satisfies a Mihlin--H\"ormander condition of sufficiently high order, then the function is an $L^p$-multiplier. A multiplier result with less restrictive assumptions on the multiplier was obtained in \cite{Giulini-Mauceri-Meda.Crelle} (and generalized to Damek--Ricci spaces in \cite{Astengo}), and it is the purpose of the present paper to establish a `spherical' counterpart to both papers in the context of Jacobi analysis. Giulini et.al. observed that there exists a function $\omega$, holomorphic and non-vanishing in a neighborhood of $\Omega_1$ such that $m$ is still a multiplier if merely the nontangential boundary value of $\omega m$ satisfies Mihlin--H\"ormander conditions, so a class of multipliers \emph{larger} than the one considered by Anker is thereby allowed. Additional remarks are to be found in the Introduction and Section 2 of \cite{Giulini-Mauceri-Meda.Crelle}. In essence $m$ is allowed to be less regular at infinity, in particular be non-integrable. This extension was not investigated in \cite{BloomXu} but our results generalize to that setting.

The precise formulation is as follows. Let $\omega(\lambda)=(\lambda^2+4\rho^2)^{\alpha+1/4}$.

\begin{theorem}\label{thm.multi-edge}
Let $m$ be an even, holomorphic function on $\Omega_1$. If $\omega m$ is
bounded on $\Omega_1$ and its nontangential boundary value $(\omega m)_\rho$ at the upper boundary line $\{\lambda+i\rho\}$ of $\Omega_1$ belongs to $\mathcal{M}_p(\R)$ for
some $p\in(1,\infty)$, then $m$ is an $L^p$-multiplier for the Jacobi
transform, and there exists a finite constant $c$ such that
$\|m\|_{\mathcal{M}_p}\leq c\|(\omega m)_\rho\|_{\mathcal{M}_p(\R)}$.
\end{theorem}

Here we adhere to the following notation and terminology: Denote by $\mathcal{M}_p(\R)$ the space of Euclidean multipliers and by $\mathcal{M}_p$ the space of Jacobi multipliers. The \emph{multiplier norm} of a function $m$ is by convention the operator norm of $f\mapsto\mathcal{F}^{-1}(\mathcal{F}f\cdot m)$ acting on $L^p(\R)$, and similarly for the Jacobi multipliers. These choices of norm turn $\mathcal{M}_p(\R)$ and $\mathcal{M}_p$ into Banach spaces. Let $d\mu(t)=(2\sinh t)^{2\alpha+1}(2\cosh t)^{2\beta+1}dt$ (the significance of this measure is explained in Section~\ref{sec.Jacobi}) and denote by $\mathcal{CO}_p^q(d\mu)$ the space of all linear operators that map boundedly from $L^p(d\mu)$ to $L^q(d\mu)$ and commute with (left) translation. The relevant translation is introduced in Equation \eqref{eqn.translation} below. We write $\mathcal{CO}_p$ instead of $\mathcal{CO}_p^p(d\mu)$, whereas the Euclidean analogue shall always be denoted by $\mathcal{CO}_p(\R)$. It is standard that every operator $T\in \mathcal{CO}_p^p(d\mu)$ has the form $Tf=k\star f$ for a unique, suitable function $k$, and where $\star$ is a suitable convolution (see Equation~\eqref{eq.convolution}). By a slight abuse of terminology we say that a function $k$ belongs to $\mathcal{CO}_p^q(d\mu)$ if the associated convolution operator $f\mapsto k\star f$ is $L^p-L^q$ bounded, hence in $\mathcal{CO}_p^q(d\mu)$.
\smallskip

The proof will follow closely the approach in \cite{Giulini-Mauceri-Meda.Crelle} and \cite{Astengo} with one crucial difference (and several smaller technical ones). We cannot use the Herz restriction principle, as we do not have any natural subgroups to which we restrict multipliers. In the present setup transference is the proper replacement, as was also utilized in both \cite{Stanton-Tomas} and \cite{Johansen-exp1}. The transference result is from \cite{Gigante} and it must be pointed out that the proof of the transference theorem is much more difficult than the version used in \cite{Stanton-Tomas}, where group-invariance of the convolution kernel may be exploited. An important realization is that the use of the Fig\`a-Talamanca--Herz algebra $A_p(\R)$ in \cite{Giulini-Mauceri-Meda.Crelle} is still permissible in the Jacobi setting, once we have transferred the analysis of the Jacobi multipliers to an Euclidean setting. We refer the reader to \cite{Cowling-Herz} for details on the Fig\`a-Talamanca--Herz algebra $A_p(\R)$ as well as the Herz restriction principle, and to \cite{CoifmanWeiss} for further details on transference.

Of a more technical level, we mention new $\cfct$-function estimates (necessitated by $\alpha,\beta$ not being half-integers), the details are summarized in Lemma \ref{lemma.precise-c}. Estimates involving the density $\Delta(t)$ also tend to become more complicated.
\bigskip

A word on notation: Error terms are always denoted by $E$ or $e$, sometimes with indices, like $E_1$ and $E_{1,1}$. This is not to imply that the different terms are somehow related, rather it is a matter of notational convenience. The notation $a\lesssim b$ is used as shorthand for an estimate of the form $a\leq cb$ for some constant $c$; this constant $c$ might change from line to line. We write out the actual constants if they are important for the conclusion.

\section{Jacobi Analysis}\label{sec.Jacobi}
In this section we briefly collect the pertinent definitions and facts relevant for Jacobi analysis. A much more detailed account can be found in \cite{Koornwinder-book}, for example. Let $(a)_0=1$ and
$(a)_k=a(a+1)\cdots (a+k-1)$. The hypergeometric function ${_2}F_1(a,b;c,z)$ is
defined by
\[{_2}F_1(a,b;c,z)=\sum_{k=0}^\infty\frac{(a)_k(b)_k}{(c)_kk!}z^k,\quad\vert z\vert<1;\]
the function $z\mapsto{_2}F_1(a,b;c,z)$ is the unique solution of the differential equation
\[z(1-z)u''(z)+(c-(a+b+1)z)u'(z)-abu(z)=0\]
which is regular in $0$ and equals $1$ there. The Jacobi function with
parameters $(\alpha,\beta)$ (which will assumed to be \emph{real}) is defined by
$\varphi_\lambda^{(\alpha,\beta)}(t)={_2}F_1(\frac{1}{2}(\alpha+\beta+1-i\lambda),
\frac{1}{2}(\alpha+\beta+1+i\lambda); \alpha+1,-\sinh^2t)$. For
$\vert\beta\vert<\alpha+1$, the system
$\{\varphi_\lambda^{(\alpha,\beta)}\}_{\lambda\geq 0}$ is a continuous
orthonormal system in $\R_+$ with respect to the weight
$\Delta_{\alpha,\beta}(t)=(2\sinh t)^{2\alpha+1}(2\cosh t)^{2\beta+1}$, $t>0$. Assume that $\alpha\neq -1,-2,\ldots$, $\alpha>\frac{1}{2}$, and $\alpha>\beta>-\frac{1}{2}$. The Jacobi-Laplacian is the operator
$\mathcal{L}=\mathcal{L}_{\alpha,\beta}=\frac{d^2}{dt^2}+((2\alpha+1)\coth
t+(2\beta+1)\tanh t)\frac{d}{dt}$, by means of which the Jacobi function
$\varphi_\lambda^{(\alpha,\beta)}$ may alternatively be characterized as the
unique solution to
\begin{equation}\label{eqn.eigeneqn}
\mathcal{L}_{\alpha,\beta}\varphi+(\lambda^2+\rho^2)\varphi=0
\end{equation}
on $\R_+$ satisfying $\varphi_\lambda(0)=1$ and $\varphi_\lambda'(0)=0$. It is
thereby clear that $\lambda\mapsto\varphi_\lambda(t)$ is analytic for all
$t\geq 0$. Moreover, for $\mathrm{Im}\,\lambda\geq 0$, there exists a unique
solution $\phi_\lambda$ to the same equation satisfying
$\phi_\lambda(t)=e^{(i\lambda-\rho)t}(1+o(1))$ as $t\to\infty$, and
$\lambda\mapsto\phi_\lambda(t)$ is therefore also analytic for $t\geq 0$.

In analogy with the case of symmetric spaces, one proceeds to show the existence of a function
$\cfct=\cfct_{\alpha,\beta}$ for which
$\varphi_\lambda(t)=\cfct(\lambda)e^{(i\lambda-\rho)t}\phi_\lambda(t)+\cfct(-\lambda)e^{(-i\lambda-\rho)t}\phi_{-\lambda}(t)$.
Since we adhere to the conventions and normalization used in
\cite{Koornwinder-FJ}, the $\cfct$-function is given by
\[\cfct(\lambda)=\frac{2^{\rho-i\lambda}\Gamma(i\lambda)\Gamma(\alpha+1)}
{\Gamma(\frac{1}{2}(\rho+i\lambda))\Gamma(\frac{1}{2}(\rho+i\lambda)-\beta)}.\]
Observe that for $\alpha,\beta\neq -1,-2,\ldots$, $\cfct(-\lambda)^{-1}$ has
finitely many poles for $\mathrm{Im}\,\lambda <0$ and none if $\mathrm{Im}\,\lambda\geq 0$ and $\mathrm{Re}\,\rho>0$. It follows from Stirling's formula that
for every $r>0$ there exists a positive constant $c_r$ such  that
\begin{equation}\label{eqn.c-fct.est}
\vert\cfct(-\lambda)\vert^{-1}\leq c_r(1+\vert\lambda\vert)^{\mathrm{Re}\,\alpha+\frac{1}{2}}\text{ if } \mathrm{Im}\,\lambda\geq 0\text{ and }
\cfct(-\lambda')\neq 0\text{ for }\vert\lambda'-\lambda\vert\leq r.
\end{equation}

\begin{lemma}\label{lemma.precise-c} Assume $\alpha>\beta>-\frac{1}{2}$.
 \begin{enumerate}
 \romannum
 \item For every integer $M$ there exist constants $c_i, i=0,\ldots,M-1$ (depending on $\alpha$, $\beta$, and $M$) such that
     \[\vert\cfct(\lambda)\vert^{-2}\thicksim c_0\vert\lambda\vert^{2\alpha+1}\Biggl\{1+\sum_{j=1}^{M-1}c_j\lambda^{-j}+O\bigl(\lambda^{-M}\bigl)\Biggr\}\text{ as } \vert\lambda\vert\to\infty.\]
\item Let $\dfct(\lambda)=\vert\cfct(\lambda)\vert^{-2}$,
$\lambda\geq 0$, and $k\in\N_0$. There exists a constant
$c_k=c_{k,\alpha,\beta}$ such that
\[\Bigl|\frac{d^k}{d\lambda^k}\dfct(\lambda)\Bigr|\leq
c_k(1+\vert\lambda\vert)^{2\alpha+1-k}.\]
\item $\cfct'(\lambda)\thicksim \cfct(\lambda)O(\lambda^{-1})$ and $\cfct''(\lambda)\thicksim\cfct(\lambda)O(\lambda^{-2})$.
\end{enumerate}
\end{lemma}
In particular $\bigl|\frac{d}{d\lambda}\cfct(\lambda)^{-1}\bigl| = \bigl|\cfct(\lambda)^{-2}\cfct'(\lambda)\bigr|\lesssim \bigl|\cfct(\lambda)^{-1}\frac{1}{\lambda}\bigr|\lesssim \vert\lambda\vert^{\alpha-\frac{1}{2}}$ for $\vert\lambda\vert$ large.	
\begin{proof}
We refer the reader to \cite[Lemma~2.1]{Johansen-disc} for a proof. This improves on the usual asymptotic statement that $\vert\cfct(\lambda)\vert^{-2}\thicksim \vert\lambda\vert^{2\alpha+1}$ as $\vert\lambda\vert\to\infty$, cf. \cite[Lemma~4.2]{Stanton-Tomas}.
\end{proof}

Let $d\nu(\lambda)=d\nu_{\alpha,\beta}(\lambda)=(2\pi)^{-\frac{1}{2}}\vert\cfct(\lambda)\vert^{-2}\,d\lambda$ and denote by $L^p(d\nu)$ the associated weighted Lebesgue space on $\R_+$; note that $\cfct(\lambda)\cfct(-\lambda)=\cfct(\lambda)\overline{\cfct(\lambda)}=\vert\cfct(\lambda)\vert^2$ whenever $\alpha,\beta,\lambda\in\R$. The Jacobi transform, initially defined for
$f\in C_c^\infty(\R_+)$ by
\[\widehat{f}(\lambda)=\frac{\sqrt{\pi}}{\Gamma(\alpha+1)}\int_0^\infty f(t)\varphi_\lambda(t)\,d\mu(t),\]
extends to a unitary isomorphism from $L^2(d\mu)$ onto $L^2(d\nu)$,
and the inversion formula is the statement that
\[f(t)=\int_0^\infty\widehat{f}(\lambda)\varphi_\lambda(t)\,d\nu(\lambda)\]
holds in the $L^2$-sense, cf. \cite[Formula~4.5]{Koornwinder-newproof}. The limiting case $\alpha=\beta=-\frac{1}{2}$ is the
Fourier-cosine transform, which we will not study. One easily verifies that
$\widehat{\mathcal{L}f}(\lambda)=-(\lambda^2+\rho^2)\widehat{f}(\lambda)$.

\begin{remark}\label{remark.GK}
For special values of $\alpha$ and $\beta$, determined by the root system of a
rank one Riemannian symmetric space, the functions $\varphi_\lambda$ are the
usual spherical functions of Harish-Chandra. To be more precise
assume $G/K$ is a rank one Riemannian symmetric space of noncompact type, with
positive roots $\alpha$ and $2\alpha$. Furthermore let $p$ denote the
multiplicity of $\alpha$ and $q$ the multiplicity of $2\alpha$ (we allow $q$ to
be zero). With $\alpha:=\frac{1}{2}(p+q-1)$ and $\beta:=\frac{1}{2}(q-1)$ both real, and $p=2(\alpha-\beta)$ and $q=2\beta+1$, the
function $\varphi^{(\alpha,\beta)}_\lambda$ is precisely the usual elementary
spherical function $\varphi_\lambda$ as considered by Harish-Chandra, $\rho=\alpha+\beta+1=\frac{1}{2}(p+2q)$ as it should be, and $\text{dim}(G/K)=p+q+1=2\alpha+2$.

A similar choice of parameters $\alpha,\beta$ reveals that even spherical
analysis on Damek--Ricci spaces is subsumed by the present setup. This was exploited in \cite{Anker-Damek-Yacoub}. One should also observe that Jacobi analysis can (perhaps should) be placed in the framework of harmonic analysis of hypergeometric functions associated to root systems; according to \cite[p.~89f]{Opdam}, the hypergeometric functions for a rank one root system with non-negative multiplicity function $k$ (the
construction of which is explained, for example, in \cite{GesturHenrik})
are then expressed by
\[F(\lambda,k,t):={_2}F_1\Bigl(\frac{\lambda+\rho}{2},\frac{-\lambda+\rho}{2},k_1+k_2+\frac{1}{2},-\sinh^2t\Bigr).\]
These are special types of Jacobi functions; with $\alpha=k_1+k_2-\frac{1}{2}$,
and $\beta=k_2-\frac{1}{2}$, one observes that
$F(i\lambda,k;t)=\varphi_\lambda^{(\alpha,\beta)}(t)$. The ideal situation
where $\alpha>\frac{1}{2}$, $\alpha>\beta>-\frac{1}{2}$
thus amounts to the requirement that $k_2>0$ and $k_1>1-k_2$.
\end{remark}

Recall from
\cite[Formula~(5.1)]{Koornwinder-FJ} the generalized translation $\tau_x$ of a suitable
function $f$ on $\R_+$, which is defined by
\begin{equation}\label{eqn.translation}
(\tau_xf)(y)=\int_0^\infty f(z)K(x,y,z)\,d\mu(z)
\end{equation}
where $K$ is an explicitly
known kernel function such that
\[
\varphi_\lambda(x)\varphi_\lambda(y)=\int_0^\infty\varphi_\lambda(z)K(x,y,z)\,d\mu(z).\]
In fact (cf. \cite[Formulae~(4.16),(4.19)]{Koornwinder-FJ}), for $\vert s-t\vert<u<s+t$,
\[\begin{split}
K(s,t,u)&=\frac{c_{\alpha,\beta}}{(\sinh s\sinh t\sinh u)^{2\alpha}}\int_0^\pi(1-\cosh^2s-\cosh^2t-\cosh^2u\\
&+2\cosh s\cosh t\cosh u\cosh y)_+^{\alpha-\beta-1}\sin^{2\beta}y\,dy\\
&=\frac{2^{\frac{1}{2}-\rho}\Gamma(\alpha+1)(\cosh s\cosh t\cosh u)^{\alpha-\beta-1}}
{\Gamma(\alpha+\frac{1}{2})(\sinh s\sinh t\sinh u)^{2\alpha}}\\
&\quad\times (1-B^2)^{\alpha-\frac{1}{2}}{_2}F_1\bigl(\alpha+\beta,\alpha-\beta;\alpha+\tfrac{1}{2};\tfrac{1}{2}(1-B)\bigr)
\end{split}\]
where $B(s,t,u)=\frac{\cosh^2s+\cosh^2t+\cosh^2u-1}{2\cosh s\cosh t\cosh u}$; elsewhere $K\equiv 0$. The associated generalized convolution product of two
functions $f,g\in L^2(d\mu)$ is defined by
\begin{equation}\label{eq.convolution}
f\star g(x)=\int_0^\infty
f(y)(\tau_xg)(y)\,d\mu(y) = \int_0^\infty f(y)g(z)K(x,y,z)\,d\mu(z)\,d\mu(y).
\end{equation}
This convolution is
associative and distributive, and by \cite[Equation~(5.4)(iv)]{Koornwinder-FJ},
$\widehat{f\star g}(\lambda)=\widehat{f}(\lambda)\widehat{g}(\lambda)$. The
usual inequalities for convolutions continue to hold, as we have the following
general form of the Young inequality.

\begin{proposition}\label{prop.young.ineq}
Let $p,q,$ and $r$ be such that $1\leq p,q,r\leq\infty$ and
$\frac{1}{p}+\frac{1}{q}-1=\frac{1}{r}$. The convolution $f\star g$ of $f\in
L^p(d\mu)$ and $g\in L^q(d\mu)$ is then well-defined as a function in
$L^r(d\mu)$, and $\|f\star g\|_r\leq\|f\|_p\|g\|_q$.
\end{proposition}
\begin{proof}
See \cite[Theorem~5.4]{Koornwinder-FJ}.
\end{proof}

\begin{definition}\label{def.multiplier}
Let $m$ be a bounded, measurable, even function on $\R$, and let $T_m$ be the
bounded linear operator defined for $f\in L^2(d\mu)$ by $\widehat{T_mf}(\lambda)=m(\lambda)\widehat{f}(\lambda)$, $\lambda\in\R$. The function $m$ is called an \emph{$L^p$-multiplier for the Jacobi transform},
with $p\in(1,\infty)$, if the operator $T_m$ extends from $L^2(d\mu)\cap
L^p(d\mu)$ to a bounded linear operator on $L^p(\R_+,d\mu)$.
\end{definition}

\begin{remark}\label{Remark.even-odd}
The multiplier results in \cite{Astengo} and \cite{Giulini-Mauceri-Meda.Crelle} are formulated for operators acting on functions that are not necessarily radial. The analogly in Jacobi analysis would be to consider functions on $\R$ that are not necessarily even, and our main theorem can be reformulated accordingly as follows. Write a function $f$ on $\R$ as the sum of its even and odd parts, $f=f_e+f_o$, and notice that one can still define the convolution between an even and an odd function. One verifies that $K(-x,y,z)=(-1)^{2\alpha}K(x,y,z)$ for all $x,y,z>0$, so that $\vert K(-x,y,z)\vert=\vert K(x,y,z)\vert$ and correspondingly $\vert K(x,y,z)f(x)\vert^p\leq 2^p \vert K(x,y,z) f_e(x)\vert^p$. The norm of $k\star f$ ($k$ still even) as an element of $L^p(\R,d\mu)$ is therefore controlled by the norm of $k\star f_e$, which is in $L^p(\R_+,d\mu)$. While this extension is straightforward, it is also cumbersome to write all the time.  All statements to follow can be modified to be about $L^p(\R,d\mu)$ rather than $L^p(\R_+,d\mu)$ but since one cannot naturally identify $L^p(\R)$ with $L^p(G/K)$ in the case of $\alpha,\beta$ being geometric, we do not obtain statements about operators acting on $L^p(G/K)$. While going from $L^p(\R_+)$ to $L^p(\R)$ in Jacobi analysis is straightforward, the same cannot be said about $L^p(K\setminus G/K)$ versus $L^p(G(K)$.
\end{remark}
\section{Local Analysis}
We prove Theorem \ref{thm.multi-edge} by separately investigating the local and
the global part of the kernel. Fix a smooth, even function $\psi$
on $\R$ such that $0\leq\psi\leq 1$, $\psi(t)\equiv 1$ for $\vert t\vert\leq
R_0^{1/2}$, and $\psi(t)\equiv 0$ for $\vert t\vert\geq R_0$, where $R_0$ is
the constant from \cite[Lemma~3.1]{Johansen-disc}\footnote{For $\alpha>\frac{1}{2}$, $\alpha >\beta>-\frac{1}{2}$, and suitable $\lambda$ there exist constants $R_0, R_1\in
(1,\sqrt{\frac{\pi}{2}})$ with $R_0^2<R_1$ such that for every $M\in\N$ and
every $t\in[0,R_0]$
\begin{eqnarray}
\label{eqn.expansion.full}\varphi_\lambda^{(\alpha,\beta)}(t) =
\frac{2\Gamma(\alpha+1)}{\Gamma(\alpha+\tfrac{1}{2})\Gamma(\tfrac{1}{2})}\frac{t^{\alpha+\frac{1}{2}}}{\sqrt{\Delta(t)}} \sum_{m=0}^\infty
a_m(t)t^{2m}\mathcal{J}_{m+\alpha}(\lambda t)
\\
\varphi_\lambda^{(\alpha,\beta)}(t) =
\frac{2\Gamma(\alpha+1)}{\Gamma(\alpha+\tfrac{1}{2})\Gamma(\tfrac{1}{2})}\frac{t^{\alpha+\frac{1}{2}}}{\sqrt{\Delta(t)}} \sum_{m=0}^M
a_m(t)t^{2m}\mathcal{J}_{m+\alpha}(\lambda t)+E_{M+1}(\lambda t),
\end{eqnarray}
with good estimates on the error term $E_{M+1}$ and the functions $a_m$.}, see also \cite[Theorem~2.1]{Stanton-Tomas}. Let $k$ be the inverse
Jacobi transform of the multiplier function $m$, regarded as an even
distribution on $\R$. As in \cite{Giulini-Mauceri-Meda.Crelle}, we cancel out
possible poles by introducing the modified multiplier function
$M(\lambda):=m(\lambda)\cfct(-\lambda)^{-1}$, $\lambda\in\R$. Since $M$ extends
to a function that is holomorphic in $\Omega_1$ and bounded on strips of the form
$\{z\in\C\,:\,\varepsilon-\rho\leq\mathrm{Im}\, z<\rho\}$, $\varepsilon>0$, the Fatou lemma guarantees that $M$ has a nontangential limit $M_\rho$ at almost every point
of the line $\{\lambda+i\rho\,:\,\lambda\in\R\}$.

\begin{proposition}\label{prop.local}
Let $m$ be an even function on $\R$ with the property that $M$ belongs to $\mathcal{M}_p(\R)$ for some $p\in(1,\infty)$, and let $k=m^\vee$. Then $\psi k\in \mathcal{CO}_p$ and $\|\psi k\|_{\mathcal{CO}_p}\lesssim \|M\|_{\mathcal{M}_p(\R)}$.
\end{proposition}

\begin{remark}\label{remark.approx}
In what follows we will assume without loss of generality that the multiplier function $m$ be rapidly decreasing. The reduction to this special situation is based on a standard use of heat kernel techniques, already indicated in \cite[Remark~1, p.266]{Stanton-Tomas} and made more precise in the proof of \cite[Proposition~4.3]{Astengo}. Let us briefly recall the technique.

Let $m$ be an arbitrary bounded measurable function on $\R$ and define $m_t(\lambda)=m(\lambda)e^{-t(\lambda^2+\rho^2)}$ for $t\geq 0$, with inverse Jacobi transform being given by $k_t=h_t\star m^\vee$, where $h_t$ is the heat kernel corresponding to $e^{t\mathcal{L}_{\alpha,\beta}}$ on $\R$. The functions $m_t$ are rapidly decreasing and form an approximate identity, since the Jacobi heat semigroup is ultracontractive. This is an easy calculation:  The Jacobi
transform being a unitary map from $L^2(d\mu)$ to $L^2(d\nu)$, we
conclude that
\[\begin{split}
\|h_t\|_{L^2(d\mu)} &= \|\widehat{h_t}\|_{L^2(d\nu)} =
\biggl(\frac{1}{2\pi}\int_0^\infty
\bigl|e^{-t(\lambda^2+\rho^2)}\bigr|^2\vert\cfct(\lambda)\vert^{-2}\,d\lambda\biggr)^{\frac{1}{2}}\\
&= e^{-t\rho}\biggl(\frac{1}{2\pi}\int_0^\infty
e^{-2t\lambda^2}\vert\cfct(\lambda)\vert^{-2}\,d\lambda\biggr)^{\frac{1}{2}}\lesssim e^{-t\rho}\biggl(\frac{1}{2\pi}\int_0^\infty
e^{-2t\lambda^2}(1+\lambda)^{2\alpha}\,d\lambda\biggr)^{\frac{1}{2}}\\ &\lesssim e^{-t\rho}.
\end{split}\]

Moreover $\|M_t\|_{\mathcal{CO}_p(\R)} = \|M\|_{\mathcal{CO}_p(\R)}$ for all $t>0$, so once Proposition \ref{prop.local} has been established for rapidly decreasing kernels, the inequality $\|\psi k_t\|_{\mathcal{CO}_p}\lesssim \|M_t\|_{\mathcal{CO}_p(\R)}= \|M\|_{\mathcal{CO}_p(\R)}$ holds for all $t>0$ as well. But then $\sup_{t>0}\|\psi k_t\|_{\mathcal{CO}_p}\lesssim\|M\|_{\mathcal{CO}_p(\R)}$, implying that
$\|\psi k\|_{\mathcal{CO}_p}=\lim_{t\to 0^+}\|\psi k_t\|_{\mathcal{CO}_p}\lesssim \|M\|_{\mathcal{CO}_p(\R)}$.
\end{remark}
\begin{proof}[Proof of Proposition~\ref{prop.local}]
We may assume by duality that $p\in(1,2]$ and by Remark \ref{remark.approx} that
$m$ is rapidly decreasing. By
the inversion formula for the Jacobi transform, the kernel $k$ may thus be
written as $k(t)=\int_0^\infty m(\lambda)\varphi_\lambda(t)\,d\nu(\lambda)$ for $t\geq
0$. For the present proof it
suffices to terminate the asymptotic expansion of $\varphi_\lambda$ from \cite[Lemma~3.1]{Johansen-disc} after two terms (corresponding to the case $M=1$): Write $\mathcal{J}_\alpha(\lambda)=(\lambda)^{-\alpha}J_\alpha(\lambda)$, where $J_\alpha$ is the usual second order Bessel function of order $\alpha$. Then
\[\varphi_\lambda(t)=c_\alpha\frac{t^{\alpha+\frac{1}{2}}}{\sqrt{\Delta(t)}}\bigl(a_0(t)\mathcal{J}_\alpha(\lambda
t)+a_1(t)t^2\mathcal{J}_{\alpha+1}(\lambda t)+E_2(\lambda t)\bigr),\] where
$a_0(t)\equiv 1$, $\vert a_1(t)\vert\lesssim R_1^{-(\alpha+\frac{1}{2})}$,
$\vert E_2(\lambda t)\vert\lesssim t^4$ if $\vert\lambda
t\vert\leq 1$, and $\vert E_2(\lambda t)\vert\lesssim t^4\vert\lambda t\vert^{-(\alpha+2)}$ if $\vert\lambda
t\vert\geq 1$. Correspondingly,
\[\begin{split}
\psi(t)k(t)&=c_\alpha \frac{t^{\alpha+\frac{1}{2}}}{\sqrt{\Delta(t)}}\psi(t)\biggl(\int_0^\infty
m(\lambda)\mathcal{J}_\alpha(\lambda t)\,d\nu(\lambda) +
a_1(t)t^2\int_0^\infty m(\lambda)\mathcal{J}_{\alpha+1}(\lambda
t)\,d\nu(\lambda) \\
&\qquad + \int_0^\infty m(\lambda)E_2(\lambda t)\, d\nu(\lambda)\biggr)
=: I_1(t)+I_2(t)+I_3(t).\end{split}\]

We presently analyze the contribution $I_3(t)$ from the error term $E_2$. First note that by the $\cfct$-function estimates from
\eqref{eqn.c-fct.est} and Lemma \ref{lemma.precise-c},

\begin{multline}\label{eqn.E2}
\int_0^\infty \frac{t^{\alpha+\frac{1}{2}}}{\vert \sqrt{\Delta(t)}\vert} \psi(t)\biggl|\int_\R
m(\lambda)E_2(\lambda
t)\,d\nu(\lambda)\biggr|\vert{\Delta'}(t)\vert^2\,dt\\
\lesssim
\|M\|_\infty \int_0^\infty
\frac{t^{\alpha+\frac{1}{2}}}{\vert \sqrt{\Delta(t)}\vert}\psi(t)\Biggl\{\int_{\vert\lambda t\vert\leq 1}\vert
E_2(\lambda t)\vert\vert\cfct(\lambda)\vert^{-1}\,d\lambda
 + \int_{\vert\lambda t\vert\geq 1}\vert
E_2(\lambda t)\vert\vert\cfct(\lambda)\vert^{-1}\,d\lambda\Biggr\}\,dt\\
\lesssim \|M\|_\infty\int_0^\infty
t^{\alpha+\frac{1}{2}}\vert \sqrt{\Delta(t)}\vert^{1/2}\psi(t)\Biggl\{\int_{\vert\lambda\vert\leq
\frac{1}{t}}t^4(1+\vert\lambda\vert)^{\alpha+\frac{1}{2}}\,d\lambda \\ + \int_{\vert\lambda\vert\geq
\frac{1}{t}} t^4\vert\lambda
t\vert^{-(\alpha+2)}\vert\lambda\vert^{\alpha+\frac{1}{2}}\,d\lambda\Biggr\}\,dt.
\end{multline}
The integral $\int_{\vert\lambda\vert\leq
\frac{1}{t}}t^4(1+\vert\lambda\vert)^{\alpha+\frac{1}{2}}\,d\lambda$ is finite since
$\alpha>-\frac{1}{2}$, and one computes that
\[\int_{\vert\lambda\vert\geq
\frac{1}{t}}t^4(1+\vert\lambda\vert)^{\alpha+\frac{1}{2}}\,d\lambda=2t^{-(\alpha+\frac{1}{2})}.\]
Collecting powers of $t$ in the above integral and using that $\psi$ is
compactly supported in a neighborhood around $t=0$, we conclude that the quantity in the last line of
\eqref{eqn.E2} may be bounded by $C\|M\|_\infty$. It thus follows from Proposition \ref{prop.young.ineq} that $I_3\in\mathcal{CO}_p$ for all $p\in(1,2]$.

We must also investigate the contributions $I_1$ and $I_2$, and to this end we consider the even functions defined on $\R_+$ by
$b_0(t) = \psi(t)\frac{t^{\alpha+\frac{1}{2}}}{\sqrt{\Delta(t)}}$ and $b_1(t)=\psi(t)\frac{t^{\alpha+\frac{1}{2}}}{\sqrt{\Delta(t)}}t^2a_1(t)$, $t>0$, together with the functions
\[\begin{split}
\Psi_j(t)&=b_j(t)\int_\R m(\lambda)\mathcal{J}_{\alpha+j}(\lambda t)\,d\nu(t),\qquad j=0,1\\ &=b_j(t)\Biggl\{\int_{J_t}m(\lambda)\mathcal{J}_{\alpha+j}(\lambda t)\,d\nu(\lambda) +
\int_{\R\setminus J_t}m(\lambda)\mathcal{J}_{\alpha+j}(\lambda t)\,d\nu(\lambda)\Biggr\}=:\Psi_j^0(t)+\Psi_j^\infty(t),
\end{split}\]
where $J_t=(-\frac{1}{t},\frac{1}{t})$. Observe that $I_j=\Psi_j$ for $j=0,1$. The point is that for $\lambda\in \R\setminus J_t$, say, we
have $\vert\lambda t\vert\geq 1$ and may use improved estimates for the modified Bessel function
$\mathcal{J}_{\alpha+j}(\lambda t)$ obtained in \cite[Appendix~A]{Johansen-exp1}, closely resembling those used for the proof of \cite[Theorem~2.1]{Stanton-Tomas}. We wish to prove
that $\Psi_j$ belongs to $\mathcal{CO}_p$, with convolution operator-norm
proportional with $\|M\|_{\mathcal{M}_p}$. To this end one observes that the local contributions $\Psi_j^0, j=0,1$ belong to $L^1(d\mu)$ \emph{with norm proportional with $\|M\|_\infty$}, since
\[\begin{split}
\|\Psi_j^0\|_{L^1(d\mu)} &\leq \int_0^\infty\Bigl|\psi(t)t^{\alpha+2j+\frac{1}{2}}a_j(t)\int_{J_t}M(\lambda)\mathcal{J}_{\alpha+j}(\lambda t)\cfct(\lambda)^{-1}\,d\lambda\Bigr|\sqrt{\Delta(t)}\,dt\\
&\lesssim \|M\|_\infty\int_0^\infty\psi(t)t^{\alpha+2j+\frac{1}{2}}\vert a_j(t)\vert \int_{J_t}\vert\lambda\vert^{\alpha+\frac{1}{2}}\,d\lambda\,\sqrt{\Delta(t)}\,dt\\
&\lesssim \|M\|_\infty\int_0^\infty\psi(t)t^{2j-1}\vert a_j(t)\vert\sqrt{\Delta(t)}\,dt = \|M\|_\infty\int_0^{R_0}\psi(t)t^{2j-1}\vert a_j(t)\vert\sqrt{\Delta(t)}\,dt\\
&\lesssim \|M\|_\infty \int_0^{R_0} t^{2j-1}R_1^{-(\alpha+j+\frac{1}{2})}t^{\alpha+\frac{1}{2}}\,dt \simeq \|M\|_\infty\int_0^{R_0}t^{\alpha+2j-\frac{1}{2}}\,dt
\end{split}\] which is indeed finite since $\alpha>\frac{1}{2}>0$ and $j=0,1$.

The functions $\Psi_j^0$ therefore give rise to $L^p$-bounded convolution operators satisfying the required norm estimate, so we concentrate on the
global part $\Psi_j^\infty$. According to the standard asymptotic expansion for Bessel functions,
\cite[p.~199, Formula~1]{Watson}, we write
\[\mathcal{J}_{\alpha+j}(s)\thicksim s^{-(\alpha+j+\frac{1}{2})}\left(\cos(s+\delta)-\beta_\alpha\frac{\sin(s+\delta)}{2s}+O(s^{-2})\right),\quad 1\leq s\leq\infty,\]
with $\beta_\alpha=\alpha(\alpha-1)$ and $\delta=-\frac{\alpha+j}{2}\pi$, leading to the decomposition
\[\begin{split}
\Psi_j^\infty(t) &= b_j(t)t^{-\alpha-j-\frac{1}{2}}\int_{\R\setminus J_t} m(\lambda)\lambda^{-\alpha-j-\frac{1}{2}}\cos(\lambda t+\delta)\,d\nu(\lambda)\\ &\quad-\frac{\beta_\alpha}{2}b_j(t)t^{-\alpha-j-\frac{3}{2}}\int_{\R\setminus J_t}m(\lambda)\lambda^{-\alpha-j-\frac{3}{2}}\sin(\lambda t+\delta)\,d\nu(\lambda)\\
&\quad+b_j(t)t^{-\alpha-j-\frac{1}{2}}\int_{\R\setminus J_t}m(\lambda)\lambda^{-\alpha-j-\frac{1}{2}}e_j(\lambda t)\,d\nu(\lambda)\\
&=:k_{j,0}(t)+k_{j,1}(t)+E_j(t),
\end{split}\]
where $\vert e_j(\lambda t)\vert=O(\vert\lambda t\vert^{-2})$. Let us write $k_{j,0}$ and $k_{j,1}$ slightly more systematically as

\[k_{j,k}(t)=c_kb_j(t)t^{-\alpha-j-k-\frac{1}{2}}\int_{\R\setminus J_t}m(\lambda)\lambda^{-\alpha-j-k-\frac{1}{2}}h_k(\lambda t+\delta)\,d\nu(\lambda), \quad j,k=0,1,\]
where $c_0=1$, $c_1=-\frac{\beta_\alpha}{2}$, $h_0(x)=\cos x$, and $h_1(x)=\sin x$. The error terms $E_j$ are readily estimated:
\[\begin{split}
\|E_j\|_{L^1(d\mu)} & \lesssim \int_0^\infty\Bigl|b_j(t)\int_{\R\setminus J_t}m(\lambda)\vert\lambda t\vert^{-\alpha-j-\frac{5}{2}}\,d\nu(\lambda)\Bigr|\Delta(t)\,dt\\
&\lesssim \int_0^\infty\vert b_j(t)\vert t^{-\alpha-j-\frac{5}{2}}\Bigl(\int_{\R\setminus J_t}M(\lambda)\vert\lambda\vert^{-\alpha-j-\frac{5}{2}}\vert\lambda\vert^{\alpha+\frac{1}{2}}\,d\lambda\Bigr)\Delta(t)\,dt\\
&\lesssim \|M\|_\infty\int_0^\infty\vert b_j(t)\vert t^{-\alpha-j-\frac{5}{2}} \Bigl(\int_{\R\setminus J_t}\vert\lambda\vert^{-\alpha-j-\frac{5}{2}}\vert\lambda\vert^{\alpha+\frac{1}{2}}\,d\lambda\Bigr) \Delta(t)\,dt
\\
&\lesssim \|M\|_\infty\int_0^{R_0}\psi(t)t^{\alpha+\frac{1}{2}+2j}\vert a_j(t)\vert t^{-\alpha-j-\frac{5}{2}}t^{\alpha+\frac{1}{2}}\,dt \quad\text{ since } \int_{\R\setminus J_t}\vert\lambda\vert^{-j-2}\,d\lambda<\infty\\
&\simeq \|M\|_\infty\times \begin{cases} \displaystyle \int_0^{R_0}\psi(t)t^{\alpha-\frac{3}{2}}\,dt &\text{for }j=0\\
\displaystyle \int_0^{R_0}\psi(t)t^{\alpha-\frac{1}{2}}\,dt&\text{for }j=1\end{cases}.
\end{split}\]
We thus see that the natural assumption that $\alpha$ be strictly greater than $-\frac{1}{2}$ does not lead to the desired estimate for $E_0$. \emph{Imposing the stronger requirement that $\alpha>\frac{1}{2}$ certainly solves this issue}.

The piece $k_{1,1}$ is just as easily handled; indeed,

\[\begin{split}
\|k_{1,1}\|_{L^1(d\mu)} & \lesssim \int_0^\infty\psi(t)t^{\alpha+\frac{5}{2}-\alpha-\frac{5}{2}} \Bigl(\int_{\R\setminus J_t} M(\lambda)\lambda^{-\alpha-\frac{5}{2}}\lambda^{\alpha+\frac{1}{2}}\,d\lambda\Bigr)\sqrt{\Delta(t)}\,dt
\\
&\lesssim \|M\|_\infty \int_0^{R_0}\psi(t)t^{\alpha+\frac{1}{2}}\Bigl(\int_{\R\setminus J_t} \lambda^{-2}\,d\lambda\Bigr)\,dt
\lesssim \|M\|_\infty\int_0^{R_0}\psi(t)t^{\alpha+\frac{1}{2}}\,dt
\end{split}\]
so $k_{1,1}$ is $\mu$-integrable with the correct norm estimate, thereby establishing the assertion of the Proposition in the case where $j+\gamma=2$.
\medskip

Assume $j+\gamma=1$ and fix a smooth
function $\Phi$ on $\R$ with
$0\leq\Phi\leq 1$, $\Phi\equiv 1$ on $[-R_0^{-1},R_0^{-1}]$, and $\Phi\equiv 0$ on $\R\setminus
[-2R_0^{-1},2R_0^{-1}]$. Correspondingly, write $k_{j,k}=K_{j,k}+E_{j,k}$, where
\[\begin{split}
K_{j,k}(t)&=c_k b_j(t)t^{-\alpha-j-k-\frac{1}{2}}\int_{\R\setminus J_t} m(\lambda)(1-\Phi(\lambda))\lambda^{-\alpha-j-k-\frac{1}{2}}h_k(\lambda t+\delta)\,d\nu(\lambda)\\
E_{j,k}(t) &= c_k b_j(t)t^{-\alpha-j-k-\frac{1}{2}}\int_{\R\setminus J_t} m(\lambda)\Phi(\lambda)\lambda^{-\alpha-j-k-\frac{1}{2}}h_k(\lambda t+\delta)\,d\nu(\lambda).
\end{split}\]
First observe that $\|E_{j,k}\|_{L^1(d\mu)}$ is bounded by
\begin{multline*}
\int_0^\infty \vert b_j(t)\vert t^{-\alpha-j-k-\frac{1}{2}}\Bigl( \int_{\R\setminus J_t}\vert M(\lambda)\vert\Phi(\lambda)\vert\lambda\vert^{-\alpha-j-k-\frac{1}{2}}\vert\lambda\vert^{\alpha+\frac{1}{2}}\vert h_k(\lambda t+\delta)\vert\,d\lambda\Bigr)\Delta(t)\,dt\\
\lesssim \|M\|_\infty\int_0^\infty\vert b_j(t)\vert t^{-\alpha-j-k-\frac{1}{2}}\Bigl(\int_{\R\setminus J_t}\Phi(\lambda)\vert\lambda^{-1}\vert h_k(\lambda t+\delta)\vert\,d\lambda\Bigr)\Delta(t)\,dt.
\end{multline*}
The integral in $\lambda$ is convergent since $\Phi$ has support in the set $[-2R_0^{-1},2R_0^{-1}]$. The integral in $t$ is estimated as above, leading to an upper estimate of the form $\|M\|_\infty\int_0^{R_0}\psi(t)t^{\alpha+j-k+\frac{1}{2}}\,dt$; this integral is finite since the power in $t$ is strictly greater than $-1$ due to the assumption that $j+k=1$. This proves the assertion for $E_{j,k}$ in the case where $j+k=1$, but the pieces $K_{j,k}$ cannot be treated nearly as naively. The problem is that the $\lambda$-integrand will now involve $1-\Phi(\lambda)$, which will grow towards the constant $1$ as $\lambda\in\R\setminus J_t$ increases. If we were to naively bound the function $h_k$ by one, the resulting integral would be divergent, so one must exploit the oscillatory nature of the integrand. The $K_{j,k}$ are still $\mu$-integrable as functions in $t$, since $m$ is rapidly decreasing, but this is not enough to guarantee the type of norm bound we are after.

Instead we use an idea from the proof of \cite[Lemma~5.6]{Johansen-exp1}: We will show that $\Delta K_{j,\gamma}$ is an $L^p$-convolutor for the Euclidean Fourier transform on $\R$ and then use the principle of transference to infer that $K_{j,k}$ is an $L^p$-convolutor for the Jacobi transform with a suitable estimate on its operator norm. Note in this regard that the convolution kernel in \cite[Theorem~4.1, Corollary~4.11, Corollary~4.12]{Gigante} merely has to be $\mu$-integrable. By the H{\"o}rmander--Mihlin multiplier theorem it therefore suffices to show that the function $t\mapsto\Delta(t)K_{j,k}(t)$ is smooth and bounded on $\R\setminus\{0\}$ and that $\vert t\vert \vert(\Delta K_{j,k})'(t)\vert$ is bounded on $\R\setminus\{0\}$. Due to the presence of the function $\psi$ in the definition of $K_{j,k}$ we may assume that $\vert t\vert\leq R_0$. Now consider the truncated integrals
\[I_+^R(t)=\int_{1/t}^Rm(\lambda)(1-\Phi(\lambda))\lambda^{-\alpha-\frac{3}{2}}h_k(\lambda t+\delta)\vert\cfct(\lambda)\vert^{-2}\,d\lambda,\quad R>0,\]
with
\begin{multline}\label{der.integral}
\frac{dI_+^R}{dt} = -m\bigl(\tfrac{1}{t}\bigr)\bigl(1-\Phi\bigl(\tfrac{1}{t}\bigr)\bigr)t^{\alpha+\frac{3}{2}}h_k(1+\delta)\bigl|\cfct\bigl(\tfrac{1}{t}\bigr)\bigr|^{-2}\\
+\int_{1/t}^Rm(\lambda)\bigl(1-\Phi(\lambda)\bigr)\lambda^{-\alpha-\frac{1}{2}}h_k'(\lambda t+\delta)\vert\cfct(\lambda)\vert^{-2}\,d\lambda.
\end{multline}
The integral in \eqref{der.integral} is obviously majorized by
\[\int_{1/t}^R\vert m(\lambda)\vert\bigl(1-\Phi(\lambda)\bigr)\lambda^{-\alpha-\frac{1}{2}}\lambda^{2\alpha+1}\,d\lambda\leq \int_{1/R_0}^R\vert m(\lambda)\vert\lambda^{\alpha+\frac{1}{2}}\,d\lambda,\]
which is finite and independent of $t$, since $m$ is rapidly decreasing. The same holds for derivatives with respect to $t$ of said integral. It follows that the function
\[t\mapsto\psi(t)t^{2j-1}a_j(t)\sqrt{\Delta(t)}\int_{1/t}^Rm(\lambda)\bigl(1-\Phi(\lambda)\bigr)\lambda^{-\alpha-\frac{1}{2}}h_k'(\lambda t+\delta)\vert\cfct(\lambda)\vert^{-2}\,d\lambda\]
is smooth and bounded away from $0$, since for small $t$, the factor $t^{2j-1}a_j(t)$ behaves roughly like $t^{2j-1}t^{\alpha+\frac{1}{2}}=t^{2j+\alpha-\frac{1}{2}}$ which does not blow up near $0$ as long as $\alpha\geq \frac{1}{2}$. Analogously the function
\[t\mapsto \psi(t)t^{2j-1}a_j(t)\sqrt{\Delta(t)}m\bigl(\tfrac{1}{t}\bigr)\bigl(1-\Phi\bigl(\tfrac{1}{t}\bigr)\bigr)t^{\alpha+\frac{3}{2}}h_k(1+\delta)\bigl|\cfct\bigl(\tfrac{1}{t}\bigr)\bigr|^{-2}\]
behaves roughly as $\psi(t)t^{2j-1}t^{\alpha+\frac{1}{2}}m\bigl(\tfrac{1}{t}\bigr)\bigl(1-\Phi\bigl(\tfrac{1}{t}\bigr)\bigr)t^{\alpha+\frac{3}{2}}t^{-(2\alpha+1)} = \psi(t))m\bigl(\tfrac{1}{t}\bigr)\bigl(1-\Phi\bigl(\tfrac{1}{t}\bigr)\bigr) t^{2j}$, which also remains smooth and bounded away from $0$. Since the exact same arguments hold for the analogously defined integrals $I_-^R$, we firstly conclude that $(\Delta K_{j,k})'$ is bounded and smooth away from zero, and secondly -- by similar calculations --  that $t\mapsto \vert t\vert\vert(\Delta K_{j,k})'(t)\vert$ is bounded as well. The assumptions in the H\"ormander--Mihlin multiplier theorem are therefore fulfilled.
\medskip

Finally suppose $j=\gamma=0$ and consider the function
$P_s:\lambda\mapsto(1-\Phi(\lambda))\vert\lambda\vert^{-s}\cfct(\lambda)^{-1}$, $s\in R\setminus\{0\}$. By the usual $\cfct$-function estimates $P_s$ is seen to be (smooth and) bounded on $\R\setminus\{0\}$ if $s\geq\alpha+\frac{1}{2}$. Moreover

\[\biggl|\frac{d}{d\lambda}P_s(\lambda)\biggr|\lesssim\vert\lambda\vert^{-s-1}\vert\cfct(\lambda)\vert^{-1} +\vert\lambda\vert^{-s}\biggl|\frac{d}{d\lambda}\cfct(\lambda)^{-1}\biggr|\lesssim  \vert\lambda\vert^{-s-1}\vert\lambda\vert^{\alpha+\frac{1}{2}} + \vert\lambda\vert^{-s}\vert\lambda\vert^{\alpha-\frac{1}{2}}
\lesssim\vert\lambda\vert^{-s+\alpha-\frac{1}{2}},\]
according to Lemma \ref{lemma.precise-c}, so $\lambda\mapsto\vert\lambda\vert \vert P_s'\vert$ is bounded on $\R\setminus\{0\}$ whenever $s\geq\alpha+\frac{1}{2}$. In other words (by the H\"ormander--Mihlin theorem) $P_s$ is an $L^p$-multiplier for the Fourier transform whenever $s\geq\alpha+\frac{1}{2}$.

It follows easily that $MP_{\alpha+\frac{1}{2}}$ is again an $L^p$-multiplier for the Euclidean Fourier transform: Let $T_m$ denote (as in Definition \ref{def.multiplier}) the operator associated with an Euclidean multiplier $m$, that is $T_m(f)=(\mathcal{F}^{-1}m)\star f$. Then
\[T_{MP_{\alpha+1/2}}f = (\mathcal{F}^{-1}(MP_{\alpha+1/2}))\star f=(T_M\circ T_{P_{\alpha+1/2}})f.\]

Fix a  compactly supported function $\widetilde{\psi}$ that is smooth away from $0$ and observe that the function
\[\widetilde{K}_{0,0}(t):=t^{\alpha}\widetilde{\psi}(t)\int_\R
m(\lambda)(1-\Phi(\lambda))\vert\lambda\vert^{-\alpha-\frac{1}{2}}e^{-i\lambda t}\,d\nu(\lambda)=ct^\alpha\widetilde{\psi}(t)\mathcal{F}(MP_{\alpha+\frac{1}{2}})(t)\]
defines a convolution operator that is bounded on $L^p(\R)$ (the convolution now referring to the Euclidean structure), hence yields an Euclidean $L^p$-multiplier. Its norm as an element in $\mathcal{CO}_p(\R)$ may now be estimated as in the third paragraph on page 168 in \cite{Giulini-Mauceri-Meda.Crelle}, to the effect that
 $\|\widetilde{K}_{0,0}\|_{\mathcal{CO}_p(\R)} \lesssim \|M\|_{\mathcal{M}_p(\R)}$. This is indeed allowed since the computation is purely Euclidean (no reference to any Jacobi analysis). The use of the space $A_p(\R)$ in the reference just quoted is therefore justified and may be repeated. This concludes the proof of the local part of the multiplier theorem.
\end{proof}
The use of transference in the above proof is precipitated by the lack of an analogue of the Herz restriction principle that was used in the proof of the analogous result \cite[Proposition~3.2]{Giulini-Mauceri-Meda.Crelle}. The proof thereby attains a Clerc--Stein-like flavour.

\section{Global Analysis}
We use the Harish-Chandra expansion
\[\varphi_\lambda(t)=\cfct(\lambda)e^{(i\lambda-\rho)t}\phi_\lambda(t)+\cfct(-\lambda)e^{(-i\lambda-\rho)t}\phi_{-\lambda}(t),\quad
\phi_\lambda(t):=\sum_{k=0}^\infty\Gamma_k(\lambda)e^{-2kt}\]
of the Jacobi function $\varphi_\lambda$ to  analyze the global part
of the kernel $k$, just as in \cite[Section~3]{Stanton-Tomas}, \cite{Giulini-Mauceri-Meda.Crelle}, and \cite{Astengo}.

\begin{lemma}[Gangolli estimates]
Let $D$ be either a compact subset of $\C\setminus (-i\N)$ or a set of the form
$D=\{\lambda=\xi+i\eta\in\C\,\vert\,\eta\geq -\varepsilon\vert\xi\vert\}$ for
some $\varepsilon\geq 0$. There exist positive constants $K,d$ such that
\begin{equation}\label{eqn.Gangolli.est}
\vert\Gamma_k(\lambda)\vert\leq K(1+k)^d \text{ for all }k\in\Z_+, \lambda\in
D.
\end{equation}
\end{lemma}
\begin{proof}
See \cite[Lemma~7]{FJ}.
\end{proof}

It follows that the expansion for $\phi_\lambda(t)$ converges uniformly on sets
of the form $\{(t,\lambda)\in[c,\infty)\times D\}$, where $c$ is a positive
constant. More precisely, if $\lambda\in D$, and $c>0$ is fixed, we see that
\[\forall t\geq c:\vert\phi_\lambda(t)\vert \leq\sum_{k=0}^\infty K(1+k)^de^{-2kt} \lesssim
\sum_{k=0}^\infty (1+k)^de^{-2ck}\lesssim 1,\] that is, $\phi_\lambda(t)$ is
bounded uniformly in $\lambda\in D$ for $t\geq c>0$. We will take $c=R_0$ in
later applications. Since $\lambda\mapsto\phi_\lambda(t)$ is analytic in a
strip containing the real axis, it follows as in the proof of
\cite[Lemma~7]{Meaney-Prestini} that derivatives of $\phi_\lambda$ in $\lambda$
are bounded independently of $\lambda$ as well.

Observe that
$\lambda\mapsto\cfct(-\lambda)^{-1}\Gamma_k(\lambda)$ is analytic in the half plane
$\{\lambda\in\C\,:\,\mathrm{Im}\,\lambda>-\rho\}$. The following result is an easy adaptation of
\cite[Lemma~3.3]{Giulini-Mauceri-Meda.Crelle}, the proof of which we include for completeness.

\begin{lemma}\label{lemma.gangolli-new}
The boundary value $(\Gamma_k)_\rho$ belongs to $\mathcal{M}_p(\R)$ for all $p\in(1,\infty)$, and
there exist positive constants $C,d$ such that $\|(\Gamma_k)_\rho\|_{\mathcal{M}_p(\R)}\leq Ck^d$
for all $k\geq 1$.
\end{lemma}
\begin{proof}
As in \cite{Giulini-Mauceri-Meda.Crelle}, we prove the lemma by means of the Mikhlin multiplier
theorem on $\R$. To this end we need a good uniform bound on the derivatives of $(\Gamma_k)_\rho$. The areforementioned standard Gangolli estimates  do not suffice, but it can be proved as in \cite[Lemma~3.3]{Giulini-Mauceri-Meda.Crelle} that
\begin{equation}\label{eqn.improved.Gangolli}
\sup_{\{\vert\mathrm{Im}\,\lambda\vert\leq\vert\mathrm{Re}\,\lambda\vert\}}\vert\Gamma_k(\lambda)\vert\lesssim k^d
\end{equation}
for a suitable constant $d$. The reader will have no trouble in repeating the proof, using that the root multiplicities $m_\alpha$ and $m_{2\alpha}$ (symmetric space parameters) are replaced by $2(\alpha-\beta)$ and $2\alpha+1$ (with $\alpha,\beta$ being Jacobi parameters), respectively.
\smallskip

Consider the
region $U=\{z\in\C\,\:\,\vert\mathrm{Im}\,(z-i\rho)\vert\leq\vert\mathrm{Re}\,(z-i\rho)\vert\}$,
together with the circle $\gamma:t\mapsto\frac{\vert\lambda\vert}{\sqrt{2}}e^{it}+(\lambda+i\rho)$, $t\in[0,2\pi]$, with center in $\lambda+i\rho$ and radius
$\frac{1}{\sqrt{2}}\vert\lambda\vert$ (which is completely contained in the inner of $U$). An application of the Cauchy Integral Formula together with the improved Gangolli estimates \eqref{eqn.improved.Gangolli}
yields the estimate
\begin{multline*}
\biggl|\frac{d\Gamma_k}{d\lambda}(\lambda+i\rho)\biggr|=\biggl|\frac{1}{2\pi
i}\int_\gamma\frac{\Gamma_k(z)}{(z-(\lambda+i\rho))^2}dz\biggr|
\\=\frac{1}{2\pi}\biggl|\int_0^{2\pi}\frac{\Gamma_k(\gamma(t))}{\bigl(\frac{\vert\lambda\vert}{\sqrt{2}}e^{it}\bigr)^2}\frac{\vert\lambda\vert}{\sqrt{2}}\,dt\biggr|\lesssim \frac{1}{\vert\lambda\vert}\int_0^{2\pi}\vert\Gamma_k(\gamma(t))\vert\,dt\lesssim \frac{k^d}{\vert\lambda\vert}.
\end{multline*}
The classical Mikhlin--H\"ormander multiplier theorem on $\R$ finishes the proof.
\end{proof}

\begin{proposition}\label{prop.global}
Let $m$ be an even function that is bounded and holomorphic on $\Omega_1$, and assume that $M$ and
$M_\rho$ are both in $\mathcal{M}_p(\R)$.  Then $(1-\psi)k$ is an $L^p$-multiplier for the Jacobi
transform with multiplier norm dominated by the sum of the multiplier norms of $M$ and $M_\rho$.
\end{proposition}

\begin{proof}
We may assume without loss of generality that $M$ and $M_\rho$ are rapidly decreasing, cf. Remark \ref{remark.approx}. Let
$K(t)=(1-\psi(t))k(t)\Delta(t)$. We shall use the principle of transference (\cite[Corollary~4.11, 4.12]{Gigante}) to infer that $(1-\psi)k$ is an $L^p$-multiplier for the Jacobi transform whenever $K$ is an $L^p$-multiplier for the Euclidean Fourier transform on $\R$. The strategy will be to insert the Harish-Chandra series for $\varphi_\lambda$ in the definition of $k(t)$, use a series expansion for $\Delta(t)$, and then analyze the various pieces individually.

As for $\Delta(t)$, observe that
\[\begin{split}
\Delta(t)&=(e^t-e^{-t})^{2\alpha+1}(e^t+e^{-t})^{2\beta+2}=e^{2\rho t}(1-e^{-2t})^{2\alpha+1}(1+e^{-2t})^{2\beta+1}\\
&=e^{2\rho t}\biggl\{\sum_{j=0}^{[[\alpha]]+[[\beta]]}c_je^{-2jt}\biggr\}(1-e^{-2t})^{\left<\alpha\right>}(1+e^{-2t})^{\left<\beta\right>}=:e^{2\rho t}\sum_{j=0}^{[[\alpha]]+[[\beta]]}c_j\delta_{\alpha,\beta}(t)e^{-2jt},
\end{split}\]
for suitable constants $c_j$. Here $[[\alpha]]$ and $\left<\alpha\right>$ denote the integer and the decimal part of $2\alpha+1$, respectively. Note that $[[\alpha]]+[[\beta]]=[[2\rho]]$.

Moreover (by the inversion formula for the even function $m$)
\[\begin{split}
m^\vee(t)&=\int_\R m(\lambda)\cfct(-\lambda)^{-1}e^{(i\lambda-\rho)t}\phi_\lambda(t)\,d\lambda = e^{-\rho t}\int_\R M(\lambda)\phi_\lambda(t)e^{i\lambda t}\,d\lambda\\
&=e^{-\rho t}\sum_{k=0}^\infty e^{-2kt}\int_\R M(\lambda)\Gamma_k(\lambda)e^{i\lambda t}\,d\lambda,
\end{split}\]
since the Harish-Chandra series converges uniformly in a suitable set of $\lambda$, implying the following expansion formula for $K(t)$:
\begin{equation}\label{eqn.formula-for-K}
\begin{split}
K(t) &=(1-\psi(t))e^{\rho t}\delta_{\alpha,\beta}(t)\sum_{j=0}^{[[2\rho]]}c_je^{-2jt}\sum_{k=0}^\infty\int_\R M(\lambda)\Gamma_k(\lambda)e^{i\lambda t}\,d\lambda
\\ &= (1-\psi(t))e^{\rho t}\delta_{\alpha,\beta}(t)\sum_{\ell=0}^\infty e^{-2\ell t}\sum_{j=0}^{[[2\rho]]} c_j\int_\R M(\lambda)\Gamma_{\ell-j}(\lambda)e^{i\lambda t}\,d\lambda,
\end{split}
\end{equation}
where $\Gamma_k\equiv 0$ for $k<0$ by convention (notice the index shift in the summation).
Define $a_l^+(t)=(1-\psi(t))e^{-2lt}\delta_{\alpha,\beta}(t)1_{[0,\infty)}(t)$ and $a_l^-(t)=(1-\psi(t))e^{2lt}\delta_{\alpha,\beta}(t)1_{(-\infty,0]}(t)$, both viewed as even functions on $\R$, and define (in analogy with \cite{Giulini-Mauceri-Meda.Crelle}) functions
\[b_j^\pm(t)=\int_\R M(\lambda)\Gamma_j(\lambda)e^{\pm(i\lambda+\rho)t}\,d\lambda, t\in\R, j\in\N_0,\text{ and } K_{\ell,j}(t)=a_\ell^-(t)b_{\ell-j}^-(t)+a_\ell^+(t)b_{\ell-j}^+(t).\]
A quick calculation establishes that $K=\sum_{\ell=0}^\infty\sum_{j=0}^{[[2\rho]]}c_jK_{\ell,j}$, and we now proceed to examine the individual $K_{\ell,j}$. The technique will be to view the integral defining $b_j^\pm$ as a path integral and then shift the contour of integration towards the upper edge of the strip $\Omega_1$. See Figure \ref{fig.shift}.

\begin{figure}[h]
 \includegraphics{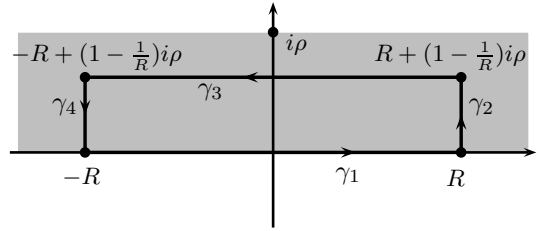}
 \vspace{-5mm}
 \caption{Change of contour-of-integration within (part of) $\Omega_1$ (in gray)}\label{fig.shift}.
\end{figure}

Set $g_\pm(\lambda)=M(\lambda)\Gamma_k(\lambda)e^{\pm(i\lambda+\rho)t}$ for fixed $k\in\N_0$, $t>0$, and parametrize the vertical segment $\gamma_2$ by $\gamma_2(s)=R+i\rho\bigl(1-\frac{1}{R}\bigr)s, s\in[0,1]$. The  $g_\pm$ are holomorphic in $\Omega_1$ and decrease rapidly as $\vert\mathrm{Re }\lambda\vert\to\infty$ with $\lambda\in\Omega_1$. Since
\begin{multline*}
\int_{\gamma_2}g_\pm(\lambda)\,d\lambda = i\rho\bigl(1-\tfrac{1}{R}\bigr)\int_0^1M\Bigl(R+i\rho\bigl(1-\tfrac{1}{R}\bigr)s\Bigr)\Gamma_k\Bigl(R+i\rho\bigl(1-\tfrac{1}{R}\bigr)s\Bigr) e^{\pm(i(R+i\rho(1-\tfrac{1}{R})s)+\rho)t}\,ds\\ =i\rho\bigl(1-\tfrac{1}{R}\bigr)e^{\pm(iR+\rho)t} \int_0^1 M\Bigl(R+i\rho\bigl(1-\tfrac{1}{R}\bigr)s\Bigr)\Gamma_k\Bigl(R+i\rho\bigl(1-\tfrac{1}{R}\bigr)s\Bigr) e^{\mp\rho(1-\tfrac{1}{R})st}\,ds
\end{multline*}
with $\bigl|\mathrm{Im}\bigl(R+i\rho(1-\frac{1}{R})s\bigr)\bigr|<\vert\rho\vert$ and $\mathrm{Re}\bigl(R+i\rho(1-\frac{1}{R})s\bigr)=R$, it follows from the improved Gangolli estimates \ref{lemma.gangolli-new} that $\bigl|\Gamma_k\bigl(R+i\rho(1-\frac{1}{R})s\bigr)\bigr|\lesssim k^d$ uniformly in $R$ (as long as $R\geq\vert\rho\vert$). Hence
\[\biggl|\int_{\gamma_2}g_\pm(\lambda)\,d\lambda\biggr|\lesssim \int_0^1\bigl| M\bigl(R+i\rho\bigl(1-\tfrac{1}{R}\bigr)s\bigr)\bigr| e^{\mp\rho(1-\frac{1}{R})st}\,dt\to 0\text{ as } R\to\infty,\] since $\vert M(z)\vert$ is rapidly decreasing in $\Omega_1$ as $\vert\mathrm{Re} z\vert\to\infty$. An analogous investigation shows that also $\bigl|\int_{\gamma_4}g_\pm(\lambda)\,d\lambda\bigr|\to 0$ as $R\to\infty$.

Parametrize the horizontal segment $\gamma_3$ by $\gamma_3(s)=\bigl(1-\frac{1}{R}\bigr)i\rho-s$, $s\in[-R,R]$. Then
\[\begin{split}
\int_{\gamma_3} g_+(\lambda)\,d\lambda
&= -\int_{-R}^R M\Bigl(\bigl(1-\tfrac{1}{R}\bigr)i\rho-s\Bigr)\Gamma_k\Bigl(\bigl(1-\tfrac{1}{R}\bigr)i\rho-s\Bigr) e^{(i((1-\frac{1}{R})i\rho-s)+\rho)t}\,ds\\
&=-\int_{-R}^RM\Bigl(\bigl(1-\tfrac{1}{R}\bigr)i\rho-s\Bigr)\Gamma_k\Bigl(\bigl(1-\tfrac{1}{R}\bigr)i\rho-s\Bigr)e^{\frac{1}{R}\rho t}e^{-its}\,ds\\
&\to -\int_\R M(i\rho-s)\Gamma_k(i\rho-s)e^{-its}\,ds\text{ as } R\to\infty\\
&=-\mathcal{F}(H_k)(t),
\end{split}\]
where $H_k(s)=M(i\rho-s)\Gamma_k(i\rho-s)$. Moreover
\[\begin{split}
\int_{\gamma_3}g_-(\lambda)\,d\lambda &=-\int_{-R}^RM\Bigl(\bigl(1-\tfrac{1}{R}\bigr)i\rho-s\Bigr)\Gamma_k\Bigl(\bigl(1-\tfrac{1}{R}\bigr)i\rho-s\Bigr) e^{-(i((1-\frac{1}{R})i\rho-s)+\rho)t}\,ds\\ &= -\int_{-R}^R M\Bigl(\bigl(1-\tfrac{1}{R}\bigr)i\rho-s\Bigr)\Gamma_k\Bigl(\bigl(1-\tfrac{1}{R}\bigr)i\rho-s\Bigr) e^{-\frac{1}{R}i\rho t}e^{its}\,ds\\
&=-\int_{-R}^R M\Bigl(\bigl(1-\tfrac{1}{R}\bigr)i\rho+s\Bigr)\Gamma_k\Bigl(\bigl(1-\tfrac{1}{R}\bigr)i\rho+s\Bigr)e^{-\frac{1}{R}\rho t}e^{-its}\,ds\\
&\to -\int_\R (M\Gamma_k)_\rho(s) e^{-its}\,ds \text{ as } R\to\infty\\
&=-\mathcal{F}((M\Gamma_k)_\rho)(t).
\end{split}\]
In other words
\begin{equation}\label{eqn.Klj}
K_{\ell,j}(t)=a_\ell^-(t)\mathcal{F}((M\Gamma_{\ell-j})_\rho)(t)+a_\ell^+(t)\mathcal{F}(H_{\ell-j})(t),
\end{equation}
where $H_{\ell-j}(t)=M(i\rho-t)\Gamma_{\ell-j}(i\rho-t)$, just as on the bottom of \cite[page~171]{Giulini-Mauceri-Meda.Crelle}.

Assuming $\ell>0$, it follows as on page 172 in \cite{Giulini-Mauceri-Meda.Crelle} that $\|a_\ell^-\|_{A_p(\R)}\leq\|a_\ell^-\|_{A_2(\R)}=\|\mathcal{F}(a_\ell^-)\|_{L^1(\R)}$. Since the $a_\ell^-$ are compactly supported, the Sobolev embedding theorem implies the estimate \[\|\mathcal{F}(a_\ell^-)\|_{L^1(\R)}\lesssim\|a_\ell^-\|_{L^\infty(\R)}\lesssim \|a_\ell^-\|_{L^2(\R)}+\|(a_\ell^-)'\|_{L^2(\R)}.\] Note that by choice of $\psi$, $1-\psi(t)\equiv 0$ for $\vert t\vert\leq R_0^{1/2}$ and $1-\psi(t)\equiv 1$ for $\vert t\vert\geq R_0>R_0^{1/2}$.

A favourable estimate for $\|a_\ell^-\|$ is obtained just as in \cite{Giulini-Mauceri-Meda.Crelle} by direct calculation:
\[\begin{split}
\|a_\ell^-\|_{L^2(\R)}^2&=\int_\R (1-\psi(t))^2e^{4\ell t}1_{(-\infty,0]}(t)\vert\delta_{\alpha,\beta}(t)\vert^2\,dt\\
&=\int_{-\infty}^0(1-\psi(t))^2\delta_{\alpha,\beta}(t)^2 e^{4\ell t}\,dt=\int_{\text{supp}(1-\psi)\cap(-\infty,0]}\delta_{\alpha,\beta}(t)^2e^{4\ell t}\,dt\\
&\lesssim\int_{-\infty}^{-R_0}\delta_{\alpha,\beta}(t)^2e^{4\ell t}\,dt\lesssim\int_{-\infty}^{-R_0}e^{4\ell t}\,dt\lesssim e^{-4\ell R_0}.
\end{split}\]
The estimate for $(a_\ell^-)'$ has no analogue in \cite{Giulini-Mauceri-Meda.Crelle}, \cite{Astengo} since the factor $\delta_{\alpha,\beta}$ is non-constant exactly when $\alpha,\beta$ are not half-integers. Its derivative must therefore be more carefully estimated. The issue is easily explained: as
\begin{multline*}
\delta_{\alpha,\beta}'(t)=2\bra(1-e^{-2t})^{\bra-1}te^{-2t}(1+e^{-2t})^{\brb}\\
-2\brb(1-e^{-2t})^{\bra}te^{-2t}(1+e^{-2t})^{\brb-1},\end{multline*}
where $-1\leq\bra-1,\brb-1<0$, the function $(1-\psi)\delta'_{\alpha,\beta}$ might not be in $L^2$. But we have an additional exponential factor of $e^{2\ell t}$ in the definition of $a_\ell^-$ that will do the trick if we estimate more carefully. To this end introduce the auxiliary function $f(t)=e^{2\ell t}\delta_{\alpha,\beta}(t)$. Then
\begin{multline*}
f'(t)=2\ell e^{2\ell t}\delta_{\alpha,\beta}(t)+2\bra(1-e^{-2t})^{\bra-1}e^{2(\ell-1)t}(1+e^{-2t})^{\brb}\\
-2\brb(1-e^{-2t})^{\bra}e^{2(\ell-1)t}(1+e^{-2t})^{\brb-1},
\end{multline*}
where
$\| \ell e^{2\ell\cdot}(1-\psi)\delta_{\alpha,\beta}\|_{L^2(\R)}^2\lesssim \ell^2\int_{-\infty}^{-R_0}e^{4\ell t}\,dt\lesssim \ell e^{-4\ell R_0}$,
which is still fine. The remaining two terms in the above expression for $f'(t)$ obviously satisfy the same type of $L^2$-estimates, except possibly when $\ell=1$, so let us assume $\ell=1$. In this case $e^{2t}\delta'_{\alpha,\beta}(t)$ behaves roughly as $e^{-2t(\bra+\brb-1)}$ for $-\infty<t\ll -R_0<0$. Since $\bra+\brb-1<1$, it follows from the presence of the mitigating additional mitigating term $2e^{2t}\delta_{\alpha,\beta}(t)$ in the expression for $f'(t)$ that there exists some positive constant $c$ such that $f'(t)\lesssim e^{-ct}$ for $-\infty<t\ll -R_0<0$, whence
$\|(a_1^-)'\|_{L^2(\R)}^2\lesssim \int_{-\infty}^{-R_0} e^{-2ct}\,dt\lesssim e^{-2cR_0}$,
with a similar bound for $\|(a_\ell^-)'\|_{L^2(\R)}$ when $\ell\geq 2$. In conclusion it has thereby been shown that
\[\sum_{\ell=1}^\infty \|a_\ell^-\|_{W^{1,1}(\R)}\leq \sum_{\ell=1}^\infty \left(\|a_\ell^-\|_{L^2(\R)}+\|(a_\ell^-)'\|_{L^2(\R)}\right)<\infty.\] The considerations for $a_\ell^+$ are similar so we shall not repeat the argument.

 By Lemma \ref{lemma.gangolli-new} it now follows that $\|K_{\ell,j}\|_{\mathcal{CO}_p(\R)}\lesssim\|a_\ell^-\|_{A_p(\R)} \|(M\Gamma_{\ell-j})_\rho\|_{\mathcal{M}_p(\R)}$, whence
 \[\biggl\|\sum_{\ell=1}^\infty\sum_{j=0}^{[[2\rho]]} c_jK_{\ell,j}\biggr\|\leq \biggl(\sum_{\ell=1}^\infty\|a_\ell^-\|_{W^{1,1}(\R)}\biggr)\|M_\rho\|_{\mathcal{M}_p(\R)} \lesssim \|M_\rho\|_{\mathcal{M}_p(\R)},\]
which is finite by assumption.
\medskip

It remains to consider the case $\ell=0$, in which case $j=0$ as well. At this point we follow the argument on page 161 in \cite{Astengo} and introduce functions \[\eta_\pm(t)=[(1-\psi(t))1_{[0,\infty)}(\pm t)-1]\delta_{\alpha,\beta}(t)e^{\mp\rho t};\] then
\[\begin{split}
K_{0,0}(t)&= a_0^-(t)b_0^-(t)+a_0^+(t)b_0^+(t)\\
&=(1-\psi(t))\delta_{\alpha,\beta}(t)1_{[0,\infty)}(t)b_0^-(t)+(1-\psi(t))\delta_{\alpha,\beta}(t)1_{(-\infty,0]}(t)b_0^+(t)\\
&=b_0^-(t)+(\eta_-(t)-1)b_0^-(t)+b_0^+(t)+(\eta_+(t)-1)b_0^+(t)\\
&=\mathcal{F}(M_\rho)(t)+\eta_-(t)\mathcal{F}(M)+\eta_+(t)\mathcal{F}(H)(t)+\mathcal{F}(H_\rho)(t),
\end{split}\]
where $H(s):=M(-s)$. The first and last summands are precisely the kernels of the multipliers $M_\rho$ and $H_\rho$, respectively. Obviously $H_\rho$ \emph{is} an $L^p$-multiplier since $M_\rho$ is one by assumption and they have the same multiplier norm. Since $\mathcal{CO}_p(\R)$ is an $A_p(\R)$-module we must now simply see that $\eta_\pm$ belong to $A_p(\R)$ but this is established by an analysis similar to the investigation of $a_\ell^\pm$ above. It thus follows that
$\|K_{0,0}\|_{\mathcal{CO}_p(\R)}\lesssim\|M_\rho\|_{\mathcal{M}_p(\R)}+
\|M\|_{\mathcal{M}_p(\R)}$.
\end{proof}
The change-of-contour technique was already used in the proof of
\cite[Proposition~4.5]{Stanton-Tomas}, see also \cite[Proposition~5.1]{Helgason-PW}, although we have altered it slightly to take into account the nontangential boundary value along the upper edge. This point wasn't stressed in \cite{Giulini-Mauceri-Meda.Crelle}, \cite{Astengo}.

There are other differences between the proof given above and the proofs of the analogous statements for rank one symmetric spaces (\cite[Proposition~3.4]{Giulini-Mauceri-Meda.Crelle}) and for Damek--Ricci spaces (\cite[Proposition~4.5]{Astengo}). Most importantly we cannot use the Herz restriction principle since there are no subgroups to which multipliers are restricted. A more technical nuisance is in regards to $\Delta(t)$: The expansion of the function $K$ into the pieces $K_{l,j}$ that was used in \cite{Giulini-Mauceri-Meda.Crelle} and \cite{Astengo} ceases to be valid in the more general setting of Jacobi analysis, since $\alpha$ and $\beta$ are no longer integers. It is insufficient to bound $K$ pointwise by $\vert K(t)\vert\leq \vert H(t)\vert$ for a suitable convolutor $H$ (where $H$ is defined as $K$ but by replacing $\Delta$ with $\sum_{j=0}^{[[2\rho]]+1}c_je^{-2jt}$), our proof is somewhat more complicated. We thank the anonymous referee on a previous version of the paper for having pointed out this problem.
\section{Proof of the Multiplier Theorem}
\begin{proof}
It suffices to prove that $M$ and
$M_\rho$ belong to $\mathcal{M}_p(\R)$, whenever $m$ satisfies the hypotheses of
Theorem \ref{thm.multi-edge}, since the conclusion will then follow from Proposition \ref{prop.local}
and Proposition \ref{prop.global}.

To this end we proceed as in \cite[p.~172--173]{Giulini-Mauceri-Meda.Crelle} and use complex
interpolation. The strategy is to `compress' the strip $\Omega_1$ and keep track of the nontangential boundary values of the modified multipliers along the edges of this compressed strip. More precisely, let $z\in\C$ with $\mathrm{Re}\, z\in [-1,1]$ and denote by $(\omega m)_{z\rho}$ the nontangential boundary value of $\omega m$ along the upper edge of $\widetilde{\Omega}_{z\rho}:=\{\lambda\in\C\,:\,\vert\mathrm{Im}\,\lambda\vert<\vert \mathrm{Re}\, z\vert\rho\}$ if $\mathrm{Re}\, z>0$ and the lower edge thereof if $\mathrm{Re}\, z<0$. This is consistent with the previously defined nontangential boundary values $(\omega m)_\rho$, in the sense that $(\omega m)_{1\cdot \rho}$ is what we previously denoted $(\omega m)_\rho$. Note that $(\omega m)_{\pm 1\cdot\rho}$ belong
to $\mathcal{M}_p(\R)$; for $x=1$ this is just the hypothesis, and for $x=-1$
this is due to the fact that $\omega m$ is even. Their respective kernels $T_{\pm 1}=\mathcal{F}^{-1}((\omega m)_{\pm \rho})$ are thus Euclidean $L^p$-convolutors. For $z\in\C$ with $\mathrm{Re}\, z\in[-1,1]$ consider the
tempered distribution $T_z$ on $\R$ that is given by $T_z=((\omega m)_{z\rho})^\vee$, and
use Euclidean convolution to define an operator $S_z$ by $S_zf=T_z\star f$, $f\in\mathcal{S}(\R)$. Clearly $S_{1+iy}$ and
$S_{-1-iy}$ extend to bounded operators on $L^p(\R)$ for all $y\in\R$, with operator norms
$|||S_{\pm(1+iy)}|||_{L^p(\R)\to L^p(\R)}=|||S_{\pm 1}|||_{L^p(\R)\to L^p(\R)}$. Complex interpolation applied
to the analytic family $\{S_z\}_{\vert\mathrm{Re}\, z\vert<1}$ proves that $S_z$ extends to a bounded operator on $L^p(\R)$ for all
$z\in(-1,1)$, such that $\widehat{T_z}=(\omega m)_{z\rho}$ belongs to $\mathcal{M}_p(\R)$ for the same range $z\in(-1,1)$.
\medskip

It thus remains to show that $m\cfct^{-1}$ and $(m\cfct^{-1})_\rho$ \emph{do} belong to
$\mathcal{M}_p(\R)$. To this end we introduce the function $\mathbf{w}:\Omega_1\to\C$,
$\mathbf{w}(\lambda)=\omega(\lambda)^{-1}\cfct(\lambda)^{-1}$ and assert that $\mathbf{w}$ and $\mathbf{w}_\rho$ satisfy H\"ormander type conditions on $\R$ of arbitrarily high order, hence define Euclidean $L^p$-multipliers. From the identity $\mathbf{w}m\omega=m\cfct^{-1}$ we infer that $m\cfct^{-1}$ and $(m\cfct^{-1})_\rho$ are indeed $L^p$-multipliers, finishing the proof.

As for the H\"ormander type estimates, note, for example, that
\[\biggl|\frac{d\mathbf{w}}{dx}\biggr| \leq \biggl|\frac{\mathbf{w}(x)}{\mathbf{w}(x)^2}\frac{d}{dx}(\cfct(x)^{-1})\biggr| +
\biggl|\frac{\mathbf{w}(x)'}{\mathbf{w}(x)^2} \cfct(x)^{-1}\biggr|
\lesssim \frac{\vert x\vert^{\alpha+\frac{1}{2}+1}\vert x\vert^\alpha}{\vert x\vert^{2\alpha}} + \frac{\vert x\vert^{\alpha+\frac{1}{2}}\vert x\vert^{\alpha-1}}{\vert x\vert^{2\alpha}} \lesssim \vert x\vert^{-\frac{1}{2}}.
\]
Additional derivatives in $x$ will produce additional decay in $\vert x\vert$; we leave the elementary details to the reader.
\end{proof}

\providecommand{\bysame}{\leavevmode\hbox to3em{\hrulefill}\thinspace}
\providecommand{\MR}{\relax\ifhmode\unskip\space\fi MR }
\providecommand{\MRhref}[2]{%
  \href{http://www.ams.org/mathscinet-getitem?mr=#1}{#2}
}
\providecommand{\href}[2]{#2}

\end{document}